\documentclass[10pt]{amsart}

\usepackage{amssymb,amsthm,amsmath}
\usepackage[numbers,sort&compress]{natbib}
\usepackage{color}
\usepackage{graphicx}
\usepackage{tikz}

\title[ ]{ANDERSON LOCALIZATION FOR THE MARYLAND MODEL WITH LONG RANGE INTERACTIONS }

\author{Jia Shi}
\address[Jia Shi]{School of Mathematical Sciences,
Fudan University,
Shanghai 200433, China} \email{15110180007@fudan.edu.cn}

\author{Xiaoping Yuan}
\address[Xiaoping Yuan]{School of Mathematical Sciences,
Fudan University,
Shanghai 200433, China} \email{xpyuan@fudan.edu.cn}

\keywords{Anderson localization, the Maryland model.}

\theoremstyle{plain}
\newtheorem{thm}{Theorem}[section]
 
 \newtheorem{lem}[thm]{Lemma}
 \newtheorem{prop}[thm]{Proposition}
 \newtheorem{rem}[thm]{Remark}
 
 \numberwithin{equation}{section}

\begin{document}


\begin{abstract}
In this paper, we establish Anderson localization for the Maryland model with long range interactions.
\end{abstract}

\maketitle
\section{Introduction and main result}

Quasi-periodic Schr\"odinger operators arise in physics. For example, we can study
\begin{equation}\label{1}
H=v_n\delta_{nn'}+\Delta,
\end{equation}
where $v_n$ is a quasi-periodic potential and $\Delta$ is the lattice Laplacian on $\mathbb{Z}$
\begin{equation*}
\Delta(n,n')=1, |n-n'|=1,\quad \Delta(n,n')=0, |n-n'|\neq 1.
\end{equation*}

Anderson localization means that $H$ has pure point spectrum with exponentially decaying eigenfunctions.
Since there are many papers on this topic, we only mention some results here. For more about dynamics and spectral theory of quasi-periodic Schr\"odinger-type operators, see the survey \cite{MJ}.

Let
\begin{equation}\label{2}
v_n=\lambda v(x+n\omega) ,
\end{equation}
where $v$ is a nonconstant real analytic potential on $\mathbb{T}$.
Fix $x=x_0$, Bourgain and Goldstein \cite{BG} proved that if $\lambda >\lambda_0$, for almost all $\omega$,
$H$ will satisfy Anderson localization. Their argument is based on a combination of large deviation estimates and general facts on semi-algebraic sets.
The method there depends explicitly on the fundamental matrix and Lyapounov exponent.
Their result is non-perturbative, which means that $\lambda_0$ does not depend on $\omega$.
By multi-scale method, Bourgain, Goldstein and Schlag \cite{BGS}  proved Anderson localization for Schr\"odinger operators on $\mathbb{Z}^{2}$
\begin{equation*}
H(\omega_1,\omega_2;\theta_1,\theta_2)=\lambda v(\theta_1+n_1\omega_1,\theta_2+n_2\omega_2)+\Delta.
\end{equation*}
Later, Bourgain \cite{B07} proved Anderson localization for quasi-periodic lattice {S}chr\"odinger operators on $\mathbb{Z}^{d}$, $d$ arbitrary.
Recently, using more elaborate semi-algebraic arguments, Bourgain and Kachkovskiy \cite{BK}
proved Anderson localization for two interacting quasi-periodic particles.

More generally, we can consider the long range model
\begin{equation}\label{3}
H=v(x+n\omega)\delta_{nn'}+\epsilon S_\phi,
\end{equation}
where $S_\phi$ is a Toeplitz operator
\begin{equation}\label{4}
S_\phi(n,n')=\hat{\phi}(n-n')
\end{equation}
and $v$ is real analytic, nonconstant on $\mathbb{T}$.
Assume $\phi$  real analytic satisfying
\begin{equation}\label{5}
  |\hat{\phi}(n)|<e^{-\rho|n|}, \quad\forall n \in \mathbb{Z}
\end{equation}
for some $\rho>0$, Bourgain \cite{B05} proved that there is $\epsilon_{0}=\epsilon_{0}(\rho)>0$, such that if
$0<\epsilon<\epsilon_{0}$, $H$ satisfies Anderson localization.
This result is non-perturbative, since $\epsilon_{0}$ does not depend on $\omega$.
Note that in the long range case, we cannot use the fundamental matrix formalism.
The method in \cite{B05} can also be used to establish Anderson localization for band Schr\"odinger operators \cite{BJ}
\begin{equation*}
H_{(n,s),(n',s')}(\omega,\theta)=\lambda v_s(\theta+n\omega)\delta_{nn'}\delta_{ss'}+\Delta,
\end{equation*}
where $\{v_s|1\leq s\leq b\}$ are real analytic, nonconstant on $\mathbb{T}$.
Recently, this method was used to prove Anderson localization for
the long-range quasi-periodic block operators \cite{JSY}
\begin{equation*}
(H(x)\vec{\psi})_n=\epsilon\sum_{k\in\mathbb{Z}}W_k\vec{\psi}_{n-k}+V(x+n\omega)\vec{\psi}_n,
\end{equation*}
where
$$V(x)=\mathrm{diag} (v_1(x),\ldots,v_l(x)),$$
$v_i(x)\  (1\leq i\leq l)$ are nonconstant real analytic functions on $\mathbb{T}$ and $W_k\  (k\in\mathbb{Z})$
are $l\times l$ matrices satisfying $W_k^*=W_{-k}$, $\|W_k\|\leq e^{-\rho|k|},\rho>0$.

Note that in the cases above, $v$ is a bounded potential and $H$ is a bounded operator.

Now let
\begin{equation}\label{6}
v_n=\lambda \tan \pi(x+n\omega) ,
\end{equation}
we have the Maryland model
\begin{equation}\label{7}
H=\lambda \tan \pi(x+n\omega)\delta_{nn'}+\Delta ,
\end{equation}
originally proposed by Grempel, Fishman and Prange \cite{GFP}.
In this case, $v_n$ is unbounded and $H$ is an unbounded operator. We will always assume
\begin{equation}\label{8}
x+n\omega-\frac{1}{2} \notin \mathbb{Z},\quad \forall n \in \mathbb{Z}
\end{equation}
to make the operator well defined. In \cite{BLS}, Bellissard, Lima and Scoppola used essentially techniques based on KAM method to prove
Anderson localization for the Maryland model.
Recently, using transfer matrix and Lyapounov exponent, Jitomirskaya and Yang \cite{JY} developed a constructive method to prove Anderson localization for the Maryland model.

More generally, we can consider the long range case of the Maryland model.
In this paper, we study
\begin{equation}\label{9}
H(x)=\tan\pi(x+n\omega)\delta_{nn'}+\epsilon S_\phi,
\end{equation}
which is the one-dimensional tight-binding model proposed by Grempel, Fishman and Prange, see Equation (1) in \cite{GFP}.

We will prove the following result:
\begin{thm}
Consider a lattice operator $H_{\omega}(x)$ of the form (\ref{9}).
Assume $\omega \in DC$ (diophantine condition),
\begin{equation}\label{11}
\|k\omega\|>c| k |^{-A},\quad \forall k\in\mathbb{Z}\setminus\{0\}
\end{equation}
and $\phi$  real analytic satisfying
\begin{equation}\label{12}
  |\hat{\phi}(n)|<e^{-\rho|n|},\quad \forall n \in \mathbb{Z}
\end{equation}
for some $\rho>0$. Fix $x_0\in\mathbb{T}$. Then there is $\epsilon_{0}=\epsilon_{0}(\rho)>0$, such that if
$0<\epsilon<\epsilon_{0}$, for almost all $\omega \in DC$, $H_{\omega}(x_0)$ satisfies Anderson localization.
\end{thm}

Our result is non-perturbative, since $\epsilon_{0}$ does not depend on $\omega$.
In the long range case here, the transfer matrix formalism is not applicable.
Our basic strategy is the same as that in \cite{B05}, but as mentioned above, the main difficulty is that
the potential $\tan$ is an unbounded function and the operator $H$ is unbounded.
In order to prove Anderson localization, we need Green's function estimates for
\begin{equation}\label{13}
G_{[0,N]}(x,E)=(R_{[0,N]}(H(x)-E)R_{[0,N]})^{-1},
\end{equation}
where $R_{\Lambda}$ is the restriction operator to $\Lambda\subset\mathbb{Z}$.
Write $\tan=\frac{\sin}{\cos}$, the singularity comes from $\frac{1}{\cos}$. Note that
\begin{equation}\label{14}
R_{[0,N]}(H(x)-E)R_{[0,N]}=A(x)B(x),
\end{equation}
where
\begin{equation}\label{15}
A(x)=\mathrm{diag}\left(\frac{1}{\cos\pi x},\ldots,\frac{1}{\cos\pi (x+N\omega)}\right).
\end{equation}
Hence
\begin{equation}\label{16}
G_{[0,N]}(x,E)=B(x)^{-1}A(x)^{-1}.
\end{equation}
In $A(x)^{-1}$, the singularity $\frac{1}{\cos}$ vanishes. This observation helps us to deal with the unbounded potential.

By Shnol's theorem \cite{H}, to establish Anderson localization for $H$, it suffices to show that if $\xi=(\xi_{n})_{n \in \mathbb{Z}}$ and $E\in\mathbb{R}$ satisfy
\begin{equation}\label{17}
|\xi_{n}|<C|n|,\quad n\rightarrow\infty,
\end{equation}
\begin{equation}\label{18}
H\xi=E\xi,
\end{equation}
then
\begin{equation}\label{19}
|\xi_{n}|<e^{-c|n|},\quad n\rightarrow\infty.
\end{equation}
Note that in our case, the operator $H$ is unbounded and the energy $E$ is unbounded.
To overcome this difficulty, we first establish Green's function estimates for energy $|E|\leq C_0$
and prove (\ref{19}) for energy $|E|\leq C_0$. Then we let $C_0\rightarrow\infty$ to obtain (\ref{19}) for all energy $E\in\mathbb{R}$.

We summarize the structure of this paper. We will prove a large deviation theorem for subharmonic functions in Section 2, which is needed for
Green's function estimates in Section 3. Then we recall some facts about semi-algebraic sets in Section 4 and give the proof of Anderson localization
in Section 5.

We will use the following notations. For positive numbers $a,b,a\lesssim b$ means $Ca\leq b$ for some constant $C>0$.
$a\ll b$ means $C$ is large. $a\sim b$ means $a\lesssim b$ and $b\lesssim a$. $N^{1-}$ means $N^{1-\epsilon}$ with some small $\epsilon>0$.
For $x\in\mathbb{R}$, $\| x\|=\inf\limits_{m\in\mathbb{Z}}|x-m|$.

\section{A large deviation theorem for subharmonic functions}

In this section, we will prove a large deviation theorem for subharmonic functions, which is needed for
Green's function estimates in Section 3.

\begin{lem}[Corollary 4.7 in \cite{B05}]\label{l2.1}
Assume $u=u(x)$ 1-periodic with subharmonic extension $\tilde{u}=\tilde{u}(z)$ to the strip $|\mathrm{Im} z|<1$
satisfying
\begin{equation}\label{a1}
|u|\leq 1, \quad|\tilde{u}|\leq B ,
\end{equation}
then
\begin{equation}\label{a2}
|\hat{u}(k)|\lesssim \frac{B}{|k|},\quad \forall k\in\mathbb{Z}\setminus\{0\}.
\end{equation}
\end{lem}

\begin{lem}[Corollary 4.10 in \cite{B05}]\label{l2.2}
Assume $u$ in Lemma \ref{l2.1} satisfying
\begin{equation}\label{a3}
\mathrm{mes}\left[x\in\mathbb{T}\Big\lvert |u(x)-\hat{u}(0)|>\epsilon_0\right]<\epsilon_1 ,
\end{equation}
then
\begin{equation}\label{a4}
\mathrm{mes}\left[x\in\mathbb{T}\Big\lvert |u(x)-\hat{u}(0)|>\sqrt{\epsilon_0}\right]<e^{-c(\sqrt{\epsilon_0}+\sqrt{\frac{\epsilon_1B}{\epsilon_0}})^{-1}}.
\end{equation}
\end{lem}

Now we can prove the following large deviation theorem.
\begin{thm}\label{t2.3}
Assume $\omega\in\mathbb{T}$ satisfies a $DC$ (diophantine condition)
\begin{equation}\label{a5}
\|k\omega\|>c| k |^{-A},\quad \forall k\in\mathbb{Z}\setminus\{0\}.
\end{equation}
Let $u:\mathbb{T}\rightarrow\mathbb{R}$ be periodic with bounded subharmonic extension $\tilde{u}$ to $|\mathrm{Im} z|\leq 1$.
Then
\begin{equation}\label{a6}
\mathrm{mes}\left[x\in\mathbb{T}\Big\lvert\left|\sum_{0\leq |m|<M}\frac{M-|m|}{M^{2}} u(x+m\omega)-\hat{u}(0)\right|>M^{-\sigma}\right]<e^{-\tilde{c}M^\sigma},\quad \tilde{c}>0
\end{equation}
for some $\sigma=\sigma(A)>0$.
\end{thm}

\begin{proof}
Let
\begin{equation}\label{a7}
v(x)=\sum_{0\leq |m|<M}\frac{M-|m|}{M^{2}} u(x+m\omega) ,
\end{equation}
then we have
\begin{equation}\label{a8}
\hat{v}(0)=\hat{u}(0),\quad v(x)-\hat{u}(0)=\sum_{k\neq 0}\hat{u}(k)\left(\sum_{0\leq |m|<M}\frac{M-|m|}{M^{2}}e^{2\pi imk\omega}\right)e^{2\pi ikx}.
\end{equation}

Since
\begin{equation}\label{a9}
\left|\sum_{0\leq |m|<M}\frac{M-|m|}{M^{2}}e^{2\pi imk\omega}\right|<\frac{1}{1+M^2\|k\omega\|^{2}},
\end{equation}
by Lemma \ref{l2.1},
\begin{equation}\label{a10}
\|v-\hat{u}(0)\|_{2}\lesssim \left[\sum_{k\neq 0}\frac{1}{|k|^{2}}\left(\frac{1}{1+M^2\|k\omega\|^{2}}\right)^{2}\right]^{\frac{1}{2}}.
\end{equation}

By (\ref{a5}),
\begin{equation}\label{a11}
\sum_{0<|k|<K}\frac{1}{|k|^{2}}\left(\frac{1}{1+M^2\|k\omega\|^{2}}\right)^{2}\lesssim\sum_{0<|k|<K}\frac{|k|^{4A-2}}{M^{4}}\lesssim \frac{K^{4A-1}}{M^{4}}.
\end{equation}

By (\ref{a10}), (\ref{a11}),
\begin{equation}\label{a12}
\|v-\hat{u}(0)\|_{2}\lesssim \left(\frac{1}{K}+\frac{K^{4A-1}}{M^{4}}\right)^{\frac{1}{2}}\leq M^{-\frac{1}{10A}},
\end{equation}
where we take $K=M^{\frac{1}{4A}}$.

By (\ref{a12}), if $\epsilon_{0}=M^{-\frac{1}{25A}}, \epsilon_{1}=M^{-\frac{3}{25A}}$, then
\begin{equation}\label{a13}
\mathrm{mes}\left[x\in\mathbb{T}\Big\lvert |v(x)-\hat{v}(0)|>\epsilon_0\right]<\epsilon_1 .
\end{equation}
By Lemma \ref{l2.2},
\begin{equation}\label{a14}
\mathrm{mes}\left[x\in\mathbb{T}\Big\lvert |v(x)-\hat{v}(0)|>M^{-\frac{1}{50A}}\right]<e^{-\tilde{c}M^{\frac{1}{50A}}} ,\quad\tilde{c}>0.
\end{equation}
This proves Theorem \ref{t2.3} if we take $\sigma=\frac{1}{50A}$.
\end{proof}

\begin{rem}\label{r2.4}
In the proof of Theorem \ref{t2.3}, we only need to assume
\begin{equation}\label{a15}
  \|k\omega\|>c|k|^{-A},\quad \forall 0<|k|\leq M.
\end{equation}
\end{rem}

\section{Green's function estimates}

In this section, we will prove Green's function estimates using the large deviation theorem in Section 2.
We will follow the method in \cite{B05}, but as mentioned in Section 1, the operator $H$ is unbounded and the energy $E$ is unbounded.
We will prove Green's function estimates for energy $|E|\leq C_0$.

\begin{prop}\label{p3.1}
Let
\begin{equation}\label{b1}
H(x)=\tan\pi(x+n\omega)\delta_{nn'}+\epsilon S_\phi.
\end{equation}
Assume $\phi$  real analytic satisfying
\begin{equation}\label{b3}
  |\hat{\phi}(n)|<e^{-\rho|n|},\quad \forall n \in \mathbb{Z}
\end{equation}
for some $\rho>0$. Then there is $\epsilon_{0}=\epsilon_{0}(\rho)>0$, such that if $0<\epsilon<\epsilon_{0}$,
the following holds:

Let $N$ be sufficiently large and
\begin{equation}\label{b4}
  \|k\omega\|>c|k|^{-A},\quad \forall 0<|k|\leq N.
\end{equation}
For energy $|E|\leq C_0$, there is $\Omega=\Omega_{N}(E)\subset\mathbb{T}$ satisfying
\begin{equation}\label{b5}
\mathrm{mes}\Omega<e^{-\tilde{c}N^{\sigma}}, \quad \sigma=\sigma(A)>0
\end{equation}
($\tilde{c}>0$ depends on $C_0$) such that if $x\notin \Omega$, then for some $|m|<\sqrt{N}$,
we have the Green's function estimate
\begin{equation}\label{b6}
|G_{[0,N)}(x+m\omega,E)(n,n')|<e^{-c_{0}(|n-n'|-\epsilon_{0}^{\frac{1}{40}}N)}, \quad n,n'\in [0,N)
\end{equation}
for some $c_{0}=c_{0}(\rho)>0$.
\end{prop}

\begin{proof}
By Cramer's rule,
\begin{equation}\label{b7}
|G_{[0,N)}(x,E)(n,n')|=\frac{|\det A_{n,n'}(x)|}{|\det[H_{N}(x)-E]|}, \quad n,n'\in [0,N)
\end{equation}
where $A_{n,n'}(x)$ refers to the ($n,n'$)-minor of $H_{N}(x)-E$.

Let
\begin{equation}\label{b8}
B_N(x)(n,n')=[\cos\pi(x+n\omega)][H_{N}(x)-E](n,n'),\quad n,n'\in [0,N),
\end{equation}
we have
\begin{equation}\label{b9}
|\det[H_{N}(x)-E]|=\left|\prod_{j=0}^{N-1}\cos\pi(x+j\omega)\right|^{-1}|\det B_N(x)|.
\end{equation}
We need to establish a lower bound for $|\det B_N(x)|$.

Since
\begin{equation}\label{b10}
B_N(x)(n,n)=\sin\pi(x+n\omega)+(\epsilon\hat{\phi}(0)-E)\cos\pi(x+n\omega),
\end{equation}
\begin{equation}\label{b11}
B_N(x)(n,n')=\epsilon\hat{\phi}(n-n')\cos\pi(x+n\omega),\quad n\neq n',
\end{equation}
the function
\begin{equation}\label{b12}
 u(x)=\frac{1}{N} \log(|\det B_N(x)|+10^{-N})
\end{equation}
admits a subharmonic extension to the complex plane, $\tilde{u}(z)$, satisfying
\begin{equation}\label{b13}
-\log 10\leq\tilde{u}(z)\leq \log\left|(\|\hat{\phi}\|_{1}+C)e^{\pi|\mathrm{Im} z|}\right|.
\end{equation}
Hence,
\begin{equation}\label{b14}
\hat{u}(0)>\int_{0}^{1}\frac{1}{N} \log|\det B_N(x)|dx=\frac{1}{N}\int_{|z|=1}\log|\det B_N(z)|, \quad z=e^{2\pi ix}.
\end{equation}
By (\ref{b10}), (\ref{b11}),
\begin{equation}\label{b15}
|\det B_N(z)|=|\det B_1(z)|,
\end{equation}
where
\begin{equation}\label{b16}
B_1(z)(n,n)=\frac{1}{2i}(e^{2\pi in\omega}z-1)+\frac{1}{2}(\epsilon\hat{\phi}(0)-E)(e^{2\pi in\omega}z+1),
\end{equation}
\begin{equation}\label{b17}
B_1(z)(n,n')=\frac{1}{2}\epsilon\hat{\phi}(n-n')(e^{2\pi in\omega}z+1),\quad n\neq n'.
\end{equation}
Since $\log|\det B_1(z)|$ is subharmonic, by Jensen inequality,
\begin{equation}\label{b18}
\int_{|z|=1}\log|\det B_1(z)|\geq\log|\det B_1(0)|.
\end{equation}
By (\ref{b16}), (\ref{b17}),
\begin{equation}\label{b19}
|\det B_1(0)|=\left(\frac{|E-i|}{2}\right)^{N}|\det [I-B_2]|,
\end{equation}
where
\begin{equation}\label{b20}
B_2(n,n')=\frac{\epsilon\hat{\phi}(n-n')}{E-i} .
\end{equation}
Since $\|B_2\|\leq \epsilon_{0}\|\hat{\phi}\|_{1}<1$, we have
\begin{equation}\label{b21}
|\det[I-B_2]^{-1}|\leq\|[I-B_2]^{-1}\|^{N}\leq (1-\epsilon_{0}\|\hat{\phi}\|_{1})^{-N}.
\end{equation}
By (\ref{b14}), (\ref{b15}), (\ref{b18}), (\ref{b19}), (\ref{b21}),
\begin{equation}\label{b22}
\hat{u}(0)>\frac{1}{2}\log (1+E^2)-\log 2 +\log(1-\epsilon_{0}\|\hat{\phi}\|_{1}).
\end{equation}

Let
\begin{equation}\label{b23}
v(x)=\sum_{0\leq |m|<M}\frac{M-|m|}{M^{2}} u(x+m\omega) ,\quad M=\sqrt{N},
\end{equation}
by Theorem \ref{t2.3} and Remark \ref{r2.4},
\begin{equation}\label{b24}
\mathrm{mes}\left[x\in\mathbb{T}\Big\lvert|v(x)-\hat{u}(0)|>N^{-\sigma}\right]<e^{-\tilde{c}N^\sigma},\quad \tilde{c},\sigma>0.
\end{equation}
Thus outside a set $\Omega=\Omega_{N}(E), \mathrm{mes}\Omega<e^{-\tilde{c}N^\sigma}$, using (\ref{b22}), we have
\begin{equation}\label{b25}
v(x)\geq\hat{u}(0)-N^{-\sigma}>\frac{1}{2}\log (1+E^2)-\log 2 +\log(1-\epsilon_{0}\|\hat{\phi}\|_{1})-N^{-\sigma}.
\end{equation}
So, for $x\notin \Omega$, there is $|m|<\sqrt{N}$, such that
\begin{equation}\label{b26}
|\det B_N(x+m\omega)|>e^{\frac{1}{2}N\log (1+E^2)-N\log 2 +N\log(1-\epsilon_{0}\|\hat{\phi}\|_{1})-N^{1-\sigma}}.
\end{equation}

Let $B_{n,n'}(x)$ be the ($n,n'$)-minor of $B_{N}(x)$, by (\ref{b8}),
\begin{equation}\label{b27}
|\det A_{n,n'}(x)|=|\cos\pi(x+n\omega)|\left|\prod_{j=0}^{N-1}\cos\pi(x+j\omega)\right|^{-1}|\det B_{n,n'}(x)|.
\end{equation}
We will obtain an upper bound on $|\det B_{n,n'}(x)|$ uniformly in $x$.

We express $\det B_{n,n'}(x)$ as a sum over paths $\gamma$ as
\begin{equation}\label{b28}
\sum_{s}\sum_{|\gamma|=s}\pm\left(\det[R_{[0,N)\backslash\gamma}B_{N}(x)R_{[0,N)\backslash\gamma}]\right)\epsilon^{s-1}
\prod_{i=1}^{s-1}\left[\hat{\phi}(\gamma_{i+1}-\gamma_{i})\cos\pi(x+\gamma_{i+1}\omega)\right],
\end{equation}
where $\gamma=(\gamma_{1},\ldots,\gamma_{s})$ is a sequence in $[0,N)$ with $\gamma_{1}=n,\gamma_{s}=n'$.

Hence
\begin{equation}\label{b29}
|\det B_{n,n'}(x)|<\sum_{s}\sum_{|\gamma|=s}\epsilon^{s-1}e^{-\rho\sum\limits_{i=1}^{s-1}|\gamma_{i+1}-\gamma_{i}|}
|\det[R_{[0,N)\backslash\gamma}B_{N}(x)R_{[0,N)\backslash\gamma}]|.
\end{equation}

If we denote $b=\sum\limits_{i=1}^{s-1}|\gamma_{i+1}-\gamma_{i}|\geq|n-n'|$ and use the fact that there are at most $2^{s-1}\binom{b}{s-1} (s,b)$-paths, then
\begin{equation}\label{b30}
|\det B_{n,n'}(x)|<\sum_{b\geq|n-n'|}\sum_{s\leq b+1}2^{s-1}\binom{b}{s-1}\epsilon^{s-1}e^{-\rho b}\max_{|\gamma|=s}|\det[R_{[0,N)\backslash\gamma}B_{N}(x)R_{[0,N)\backslash\gamma}]|.
\end{equation}

Using Hadamard inequality, we have
\begin{equation}\label{b31}
|\det[R_{[0,N)\backslash\gamma}B_{N}(x)R_{[0,N)\backslash\gamma}]|\leq\prod_{k\in[0,N)\backslash\gamma}\left[|\sin\pi(x+k\omega)-E\cos\pi(x+k\omega)|+\epsilon_{0}(\|\hat{\phi}\|_{1}+1)\right].
\end{equation}

So,
\begin{equation}\label{b32}
\log|\det[R_{[0,N)\backslash\gamma}B_{N}(x)R_{[0,N)\backslash\gamma}]|\leq\sum_{k\in[0,N)\backslash\gamma}\log[|\sin\pi(x+k\omega)-E\cos\pi(x+k\omega)|+\epsilon_{0}(\|\hat{\phi}\|_{1}+1)].
\end{equation}

Let $\alpha_{E}\in(0,1)$ such that $\sin\pi\alpha_{E}=\frac{1}{\sqrt{E^{2}+1}}, \ \cos\pi\alpha_{E}=\frac{E}{\sqrt{E^{2}+1}}$,
then by (\ref{b32}),
\begin{equation}\label{b33}
\log|\det[R_{[0,N)\backslash\gamma}B_{N}(x)R_{[0,N)\backslash\gamma}]|\leq \frac{N}{2}\log(E^{2}+1)+\sum_{k\in[0,N)\backslash\gamma}\log[|\cos\pi(x+k\omega+\alpha_{E})|+\epsilon_{0}(\|\hat{\phi}\|_{1}+1)].
\end{equation}

Let
\begin{equation}\label{b34}
S_{1}=\sum_{k\in[0,N)}\log[|\cos\pi(x+k\omega+\alpha_{E})|+\epsilon_{0}(\|\hat{\phi}\|_{1}+1)],
\end{equation}
\begin{equation*}
S_{2}=\sum_{k\in\gamma}\log[|\cos\pi(x+k\omega+\alpha_{E})|+\epsilon_{0}(\|\hat{\phi}\|_{1}+1)],\quad |\gamma|=s,
\end{equation*}
by (\ref{b33}),
\begin{equation}\label{b35}
\log|\det[R_{[0,N)\backslash\gamma}B_{N}(x)R_{[0,N)\backslash\gamma}]|\leq\frac{1}{2}N\log(E^{2}+1)+S_1-S_2.
\end{equation}

By Denjoy-Koksma type inequality (Lemma 12 in \cite{J}),
\begin{equation}\label{b36}
S_{1}\leq N \int_{0}^{1}\log[|\cos\pi x|+\epsilon_{0}(\|\hat{\phi}\|_{1}+1)]dx+N^{1-\delta},
\end{equation}
where $\delta=\delta(A)>0$.
Using $\int_{0}^{1}\log|\cos\pi x|dx=-\log2$, we have
\begin{equation}\label{b37}
\int_{0}^{1}\log[|\cos\pi x|+\epsilon_{0}(\|\hat{\phi}\|_{1}+1)]dx=-\log2+\int_{0}^{1}\log\left[1+\frac{\epsilon_{0}(\|\hat{\phi}\|_{1}+1)}{|\cos\pi x|}\right]dx.
\end{equation}

There is $C>1$, such that
\begin{equation}\label{b38}
\log(1+x)<x^{\frac{1}{2}}, \quad \forall x>C.
\end{equation}
Let
\begin{equation}\label{b39}
\eta=\epsilon_{0}(\|\hat{\phi}\|_{1}+1)<1,\quad J=\left\{x\in[0,1]\Big\lvert\frac{\eta}{|\cos\pi x|}>C\right\},
\end{equation}
then
\begin{equation}\label{b40}
J=\bigcup_{n\geq 0}J_{n},\quad J_{n}=\left\{x\in[0,1]\Big\lvert2^{n}C<\frac{\eta}{|\cos\pi x|}\leq 2^{n+1}C\right\}.
\end{equation}

Using (\ref{b38}) and the fact
\begin{equation}\label{b41}
{\rm mes}\left[x\in[0,1]\Big\lvert|\cos\pi x|<\epsilon\right]<\epsilon,\quad \forall 0<\epsilon<1,
\end{equation}
we have
\begin{equation}\label{b42}
\int_{J}\log\left(1+\frac{\eta}{|\cos\pi x|}\right)dx=\sum_{n\geq 0}\int_{J_{n}}\log\left(1+\frac{\eta}{|\cos\pi x|}\right)dx
\leq\sum_{n\geq 0}\frac{\eta}{2^{n}C}\left(2^{n+1}C\right)^{\frac{1}{2}}\leq C\eta,
\end{equation}
where $C$ refers to various constants.

Let
\begin{equation}\label{b43}
J_{-1}=\left\{x\in[0,1]\Big\lvert\eta^{\frac{1}{2}}<\frac{\eta}{|\cos\pi x|}\leq C\right\}, \quad
J_{-2}=\left\{x\in[0,1]\Big\lvert\frac{\eta}{|\cos\pi x|}\leq \eta^{\frac{1}{2}}\right\},
\end{equation}
by (\ref{b42}),
\begin{equation}\label{b44}
\int_{0}^{1}\log\left[1+\frac{\epsilon_{0}(\|\hat{\phi}\|_{1}+1)}{|\cos\pi x|}\right]dx
=\int_{J_{-2}}+\int_{J_{-1}}+\int_{J}\log\left(1+\frac{\eta}{|\cos\pi x|}\right)dx\leq C\eta^{\frac{1}{2}}<\epsilon_{0}^{\frac{1}{2}-}.
\end{equation}

Using (\ref{b36}), (\ref{b37}), (\ref{b44}), we get
\begin{equation}\label{b45}
S_{1}\leq -N\log 2+\epsilon_{0}^{\frac{1}{2}-}N .
\end{equation}

There is always the lower bound
\begin{equation}\label{b46}
S_{2}\geq s\log\epsilon_{0}.
\end{equation}

Assume $s>\epsilon_{0}^{\frac{1}{10}}N$, if $\kappa\sim \frac{s}{N}$, then by Denjoy-Koksma type inequality (Lemma 12 in \cite{J}),
\begin{equation}\label{b47}
\#\left\{k=0,\ldots,N-1\Big\lvert\|x+k\omega+\alpha_{E}-\frac{1}{2}\|<\kappa\right\}<10\kappa N.
\end{equation}
It follows that for at least $\frac{s}{2}$ elements $k\in\gamma$,
\begin{equation}\label{b48}
\log[|\cos\pi(x+k\omega+\alpha_{E})|+\epsilon_{0}(\|\hat{\phi}\|_{1}+1)]>\log\kappa^{2}>\log\epsilon_{0}^{\frac{1}{4}}.
\end{equation}
By (\ref{b48}),
\begin{equation}\label{b49}
S_{2}\geq \frac{1}{2}s\log\epsilon_{0}+\frac{1}{2}s\log\epsilon_{0}^{\frac{1}{4}}>\frac{3}{4}s\log\epsilon_{0}.
\end{equation}

By (\ref{b35}), (\ref{b45}), (\ref{b46}),
\begin{equation}\label{b50}
\log|\det[R_{[0,N)\backslash\gamma}B_{N}(x)R_{[0,N)\backslash\gamma}]|\leq\frac{1}{2}N\log(E^{2}+1)-N\log 2+\epsilon_{0}^{\frac{1}{2}-}N-s\log\epsilon_{0},
\end{equation}
and, if $s>\epsilon_{0}^{\frac{1}{10}}N$, by (\ref{b49}),
\begin{equation}\label{b51}
\log|\det[R_{[0,N)\backslash\gamma}B_{N}(x)R_{[0,N)\backslash\gamma}]|\leq\frac{1}{2}N\log(E^{2}+1)-N\log 2+\epsilon_{0}^{\frac{1}{2}-}N-\frac{3}{4}s\log\epsilon_{0}.
\end{equation}
By (\ref{b30}),
\begin{equation}\label{b52}
|\det B_{n,n'}(x)|<\sum_{b\geq|n-n'|}\sum_{s\leq b+1, s\leq\epsilon_{0}^{\frac{1}{10}}N}2^{s-1}\binom{b}{s-1}\epsilon^{s-1}
e^{-\rho b}(\frac{1}{\epsilon_{0}})^{s}e^{\frac{1}{2}N\log(E^{2}+1)-N\log 2+\epsilon_{0}^{\frac{1}{2}-}N}
\end{equation}
\begin{equation*}
+\sum_{b\geq|n-n'|}\sum_{s\leq b+1, s>\epsilon_{0}^{\frac{1}{10}}N}2^{s-1}\binom{b}{s-1}\epsilon^{s-1}
e^{-\rho b}(\frac{1}{\epsilon_{0}})^{\frac{3}{4}s}e^{\frac{1}{2}N\log(E^{2}+1)-N\log 2+\epsilon_{0}^{\frac{1}{2}-}N}.
\end{equation*}

We need to estimate
\begin{equation}\label{b53}
s_{1}=\sum_{b\geq|n-n'|}\sum_{s\leq b+1, s\leq\epsilon_{0}^{\frac{1}{10}}N}2^{s-1}\binom{b}{s-1}\epsilon^{s-1}e^{-\rho b}(\frac{1}{\epsilon_{0}})^{s},
\end{equation}
\begin{equation*}
s_{2}=\sum_{b\geq|n-n'|}\sum_{s\leq b+1, s>\epsilon_{0}^{\frac{1}{10}}N}2^{s-1}\binom{b}{s-1}\epsilon^{s-1}
e^{-\rho b}(\frac{1}{\epsilon_{0}})^{\frac{3}{4}s}.
\end{equation*}

If $|n-n'|\geq\epsilon_{0}^{\frac{1}{20}}N$, then
\begin{equation}\label{b54}
s_{1}\leq\sum_{b\geq|n-n'|}\sum_{ s\leq\epsilon_{0}^{\frac{1}{10}}N}\frac{1}{\epsilon_{0}}2^{s-1}\binom{b}{s-1}e^{-\rho b}
\leq\sum_{b\geq|n-n'|}\frac{1}{\epsilon_{0}}2^{\epsilon_{0}^{\frac{1}{10}}N}\binom{b}{\epsilon_{0}^{\frac{1}{10}}N}e^{-\rho b}.
\end{equation}
By Stierling formula,
\begin{equation}\label{b55}
\binom{b}{\epsilon_{0}^{\frac{1}{10}}N}\leq\binom{b}{\epsilon_{0}^{\frac{1}{20}}b}\leq\frac{C}{\epsilon_{0}^{\frac{1}{20}}\sqrt{N}}e^{bf(\epsilon_{0}^{\frac{1}{20}})},
\end{equation}
where
\begin{equation}\label{b56}
f(x)=-(1-x)\log(1-x)-x\log x, \quad 0<x<1.
\end{equation}
By (\ref{b54}), (\ref{b55}), (\ref{b56}),
\begin{equation}\label{b57}
s_{1}\leq\frac{1}{\epsilon_{0}}2^{\epsilon_{0}^{\frac{1}{10}}N}\frac{C}{\epsilon_{0}^{\frac{1}{20}}\sqrt{N}}\sum_{b\geq|n-n'|}e^{-\rho b+bf(\epsilon_{0}^{\frac{1}{20}})}
\leq e^{-[\rho-f(\epsilon_{0}^{\frac{1}{20}})-\epsilon_{0}^{\frac{1}{20}}\log2]|n-n'|}<e^{-\frac{\rho}{2}|n-n'|},
\end{equation}
if we take $\epsilon_0=\epsilon_0(\rho)>0$ small.

If $|n-n'|<\epsilon_{0}^{\frac{1}{20}}N$, then
\begin{equation}\label{b58}
\sum_{|n-n'|\leq b<\epsilon_{0}^{\frac{1}{20}}N}\sum_{s\leq b+1}2^{s-1}\binom{b}{s-1}\epsilon^{s-1}\frac{e^{-\rho b}}{\epsilon_{0}^{s}}
\leq\sum_{|n-n'|\leq b<\epsilon_{0}^{\frac{1}{20}}N}\frac{e^{-\rho b}}{\epsilon_{0}}(1+\frac{2\epsilon}{\epsilon_{0}})^{b}
\leq\frac{3^{\epsilon_{0}^{\frac{1}{20}}N}}{\epsilon_{0}}e^{-\rho|n-n'|}.
\end{equation}

Hence
\begin{equation}\label{b59}
s_{1}\leq e^{2\epsilon_{0}^{\frac{1}{20}}N-\rho|n-n'|}+e^{-\frac{\rho}{2}|n-n'|}.
\end{equation}

If $|n-n'|\geq\epsilon_{0}^{\frac{1}{20}}N$, then
\begin{equation}\label{b60}
s_{2}\leq \epsilon_{0}^{-\frac{3}{4}}\sum_{b\geq|n-n'|}e^{-\rho b}(1+2\epsilon_{0}^{\frac{1}{4}})^{b}\leq e^{-\frac{\rho}{2}|n-n'|}.
\end{equation}

If $|n-n'|<\epsilon_{0}^{\frac{1}{20}}N$, then
\begin{equation}\label{b61}
\sum_{|n-n'|\leq b<\epsilon_{0}^{\frac{1}{20}}N}\sum_{s\leq b+1}2^{s-1}\binom{b}{s-1}\epsilon^{s-1}\frac{e^{-\rho b}}{\epsilon_{0}^{\frac{3}{4}s}}
\leq\sum_{|n-n'|\leq b<\epsilon_{0}^{\frac{1}{20}}N}\frac{e^{-\rho b}}{\epsilon_{0}^{\frac{3}{4}}}(1+2\epsilon_{0}^{\frac{1}{4}})^{b}
\leq e^{\epsilon_{0}^{\frac{1}{20}}N-\rho|n-n'|}.
\end{equation}

Hence
\begin{equation}\label{b62}
s_{2}\leq e^{\epsilon_{0}^{\frac{1}{20}}N-\rho|n-n'|}+e^{-\frac{\rho}{2}|n-n'|}.
\end{equation}

By (\ref{b52}), (\ref{b59}), (\ref{b62}),
\begin{equation}\label{b63}
|\det B_{n,n'}(x)|<e^{\frac{1}{2}N\log(E^{2}+1)-N\log 2+\epsilon_{0}^{\frac{1}{2}-}N}(e^{2\epsilon_{0}^{\frac{1}{20}}N-\rho|n-n'|}+e^{-\frac{\rho}{2}|n-n'|}).
\end{equation}

Using (\ref{b7}), (\ref{b9}), (\ref{b26}), (\ref{b27}), (\ref{b63}),
we have for $x\notin \Omega$, there is $|m|<\sqrt{N}$, such that
\begin{equation}\label{b64}
|G_{[0,N)}(x+m\omega,E)(n,n')|<e^{N^{1-\sigma}-N\log(1-\epsilon_{0}\|\hat{\phi}\|_{1})+\epsilon_{0}^{\frac{1}{2}-}N+\epsilon_{0}^{\frac{1}{20}}N-\frac{\rho}{2}|n-n'|}
<e^{-c_{0}(|n-n'|-\epsilon_{0}^{\frac{1}{40}}N)},
\end{equation}
if we take $c_{0}=\frac{\rho}{2}$. This proves the Green's function estimate.
\end{proof}

\section{Semi-algebraic sets}

We recall some basic facts of semi-algebraic sets in this section, which is needed in Section 5. Let $\mathcal{P}=\{P_1,\ldots,P_s\}\subset\mathbb{R}[X_1,\ldots,X_n]$
be a family of real polynomials whose degrees are bounded by $d$.
A semi-algebraic set is given by
\begin{equation}\label{ss}
S=\bigcup_{j}\bigcap_{l\in L_{j}}\left\{\mathbb{R}^{n}\Big\lvert P_ls_{jl}0\right\},
\end{equation}
where $L_{j}\subset\{1,\ldots,s\},s_{jl}\in\{\leq,\geq,=\}$ are arbitrary.
We say that $S$ has degree at most $sd$ and its degree is the $\inf$ of $sd$ over all representations as in (\ref{ss}).

We need the following quantitative version of the Tarski-Seidenberg principle.
\begin{prop}[\cite{BPR}]\label{p4.1}
Let $S\subset\mathbb{R}^{n}$ be a semi-algebraic set of degree $B$, then any projection of $S$ is semi-algebraic of degree at most $B^{C}, C=C(n)$.
\end{prop}

Next fact deals with the intersection of a semi-algebraic set of small measure and the orbit of a diophantine shift.
\begin{prop}[Corollary 9.7 in \cite{B05}]\label{p4.2}
Let $S\subset[0,1]^{n}$ be semi-algebraic of degree $B$ and ${\rm mes}_{n}S<\eta$.
Let $\omega\in\mathbb{T}^{n}$ satisfy a $DC$ and
\begin{equation*}
  \log B\ll\log N\ll\log\frac{1}{\eta}.
\end{equation*}
Then for any $x_0\in\mathbb{T}^{n}$,
\begin{equation*}
  \#\{k=1,\ldots,N|x_0+k\omega\in S\}<N^{1-\delta}
\end{equation*}
for some $\delta=\delta(\omega)>0$.
\end{prop}

Finally, we will make essential use of the following transversality property.

\begin{lem}[Lemma 9.9 in \cite{B05}]\label{l4.3}
 Let $S\subset[0,1]^{2n}$ be a semi-algebraic set of degree $B$ and ${\rm mes}_{2n}S<\eta, \log B\ll\log\frac{1}{\eta}$.
We denote $(\omega,x)\in[0,1]^{n}\times[0,1]^{n}$ the product variable and $\{e_j|0\leq j\leq n-1\}$ the $\omega$-coordinate vectors.
Fix $\epsilon>\eta^{\frac{1}{2n}}$. Then there is a decomposition $S=S_1\cup S_2$,
 $S_1$ satisfying
 \begin{equation*}
{\rm mes}_{n}({\rm Proj}_\omega S_1)<B^{C}\epsilon
 \end{equation*}
 and $S_2$ satisfying the transversality property
\begin{equation*}
 {\rm mes}_{n}(S_2\cap L)<B^{C}\epsilon^{-1}\eta^{\frac{1}{2n}}
\end{equation*}
for any $n$-dimensional hyperplane $L$ such that $\max\limits_{0\leq j\leq n-1}|{\rm Proj}_L(e_j)|<\frac{\epsilon}{100}$.
\end{lem}

\section{Proof of Anderson localization}

In this section, we give the proof of Anderson localization as in \cite{BG}.

By application of the resolvent identity, we have the following
\begin{lem}\label{l5.1}
Let $I\subset\mathbb{Z}$ be an interval of size $N$ and $\{I_{\alpha}\}$ subintervals of size $M= N^{\delta}, \delta>0$ is small.
Assume $\forall k\in I$, there is some $\alpha$ such that
\begin{equation}\label{c1}
\left[k-\frac{M}{4},k+\frac{M}{4}\right]\cap I\subset I_\alpha
\end{equation}
and $\forall \alpha$,
\begin{equation}\label{c2}
|G_{I_{\alpha}}(n_1,n_2)|<e^{-c_0(|n_1-n_2|-\epsilon_{0}^{\frac{1}{40}}M)}, \quad  n_1,n_2\in I_{\alpha}.
\end{equation}
Then
\begin{equation}\label{c3}
|G_{I}(n_1,n_2)|<2e^{c_0\epsilon_{0}^{\frac{1}{40}}M}, \quad  n_1,n_2\in I ,
\end{equation}
\begin{equation}\label{c4}
|G_{I}(n_1,n_2)|<e^{-\frac{1}{2}c_0|n_1-n_2|}, \quad  n_1,n_2\in I ,|n_1-n_2|>\frac{N}{10}.
\end{equation}
\end{lem}

\begin{proof}
For $m,n\in I$, there is some $\alpha$ such that
\begin{equation}\label{c5}
\left[m-\frac{M}{4},m+\frac{M}{4}\right]\cap I\subset I_\alpha.
\end{equation}
By resolvent identity,
\begin{equation}\label{c6}
|G_{I}(m,n)|\leq e^{c_0\epsilon_{0}^{\frac{1}{40}}M}+\sum_{m_1\in I_{\alpha},m_2\notin I_{\alpha}}|G_{I_{\alpha}}(m,m_1)|e^{-\rho|m_1-m_2|}|G_{I}(m_{2},n)|.
\end{equation}
If $|m_1-m|\leq\frac{M}{8}$, then $|m_1-m_2|\geq\frac{M}{8}$, hence
\begin{equation}\label{c7}
\sum_{|m_1-m|\leq\frac{M}{8},m_2\notin I_{\alpha}}|G_{I_{\alpha}}(m,m_1)|e^{-\rho|m_1-m_2|}<M e^{-\rho\frac{M}{8}}e^{c_0\epsilon_{0}^{\frac{1}{40}}M}<\frac{1}{4}.
\end{equation}
If $|m_1-m|>\frac{M}{8}$, then $|G_{I_{\alpha}}(m,m_1)|<e^{-c_0\frac{M}{8}}e^{c_0\epsilon_{0}^{\frac{1}{40}}M}$, hence
\begin{equation}\label{c8}
\sum_{|m_1-m|>\frac{M}{8},m_2\notin I_{\alpha}}|G_{I_{\alpha}}(m,m_1)|e^{-\rho|m_1-m_2|}<M e^{-c_0\frac{M}{8}}e^{c_0\epsilon_{0}^{\frac{1}{40}}M}<\frac{1}{4}.
\end{equation}
By (\ref{c6}), (\ref{c7}), (\ref{c8}),
\begin{equation}\label{c9}
\max_{m,n\in I}|G_{I}(m,n)|<e^{c_0\epsilon_{0}^{\frac{1}{40}}M}+\frac{1}{2}\max_{m,n\in I}|G_{I}(m,n)|.
\end{equation}
(\ref{c3}) follows from (\ref{c9}).

Take $m,n\in I, |m-n|>\frac{N}{10}$, assume (\ref{c5}), by resolvent identity,
\begin{equation}\label{c10}
|G_{I}(m,n)|\leq \sum_{n_0\in I_{\alpha},n_1\notin I_{\alpha}}|G_{I_{\alpha}}(m,n_0)|e^{-\rho|n_0-n_1|}|G_{I}(n_1,n)|
\end{equation}
\begin{equation*}
  \leq Me^{c_0\epsilon_{0}^{\frac{1}{40}}M}\sum_{|m-n_1|>\frac{M}{4}}e^{-c_0|m-n_1|}|G_{I}(n_1,n)|.
\end{equation*}
Repeat the argument in (\ref{c10}), we get
\begin{equation}\label{c11}
 |G_{I}(m,n)|  \leq M^{t}e^{tc_0\epsilon_{0}^{\frac{1}{40}}M}\sum_{|m-n_1|>\frac{M}{4},\ldots,|n_{t-1}-n_t|>\frac{M}{4}}e^{-c_0(|m-n_1|+\cdots+|n_{t-1}-n_t|)}|G_{I}(n_t,n)|,
\end{equation}
where $t\leq 10\frac{N}{M}$.

If $|n-n_t|\leq M$, then by (\ref{c3}), (\ref{c11}),
\begin{equation}\label{c12}
 |G_{I}(m,n)|  \leq M^{t}N^{t}e^{tc_0\epsilon_{0}^{\frac{1}{40}}M}e^{-c_0(|m-n|-M)}2e^{c_0\epsilon_{0}^{\frac{1}{40}}M}
\end{equation}
\begin{equation*}
  \leq e^{20c_0\epsilon_{0}^{\frac{1}{40}}N+20\frac{N}{M}\log N-c_0(|m-n|-M)}\leq e^{-c_0(1-400\epsilon_{0}^{\frac{1}{40}})|m-n|}<e^{-\frac{1}{2}c_0|m-n|}.
\end{equation*}

If $t= 10\frac{N}{M}$, then by (\ref{c3}), (\ref{c11}),
\begin{equation}\label{c13}
 |G_{I}(m,n)|  \leq M^{t}N^{t}e^{tc_0\epsilon_{0}^{\frac{1}{40}}M}e^{-tc_0\frac{M}{4}}2e^{c_0\epsilon_{0}^{\frac{1}{40}}M}
\leq e^{40c_0\epsilon_{0}^{\frac{1}{40}}N-\frac{5}{2}c_0N}<e^{-2c_0N}<e^{-c_0|m-n|}.
\end{equation}

(\ref{c4}) follows from (\ref{c12}), (\ref{c13}). This proves Lemma \ref{l5.1}.
\end{proof}

Now we can prove the main result.
\begin{thm}\label{t5.2}
Consider the lattice operator $H_{\omega}(x)$ of the form
\begin{equation}\label{c14}
H_{\omega}(x)=\tan\pi(x+n\omega)\delta_{nn'}+\epsilon S_\phi.
\end{equation}
Assume $\omega \in DC$ (diophantine condition),
\begin{equation}\label{c16}
\|k\omega\|>c| k |^{-A},\quad \forall k\in\mathbb{Z}\setminus\{0\}
\end{equation}
and $\phi$  real analytic satisfying
\begin{equation}\label{c17}
  |\hat{\phi}(n)|<e^{-\rho|n|},\quad \forall n \in \mathbb{Z}
\end{equation}
for some $\rho>0$. Fix $x_0\in\mathbb{T}$. Then there is $\epsilon_{0}=\epsilon_{0}(\rho)>0$, such that if
$0<\epsilon<\epsilon_{0}$, for almost all $\omega \in DC$, $H_{\omega}(x_0)$ satisfies Anderson localization.
\end{thm}

\begin{proof}
To establish Anderson localization, it suffices to show that if $\xi=(\xi_n)_{n\in\mathbb{Z}},E\in\mathbb{R}$ satisfy
\begin{equation}\label{c18}
\xi_0=1, |\xi_n|<C|n|,\quad |n|\rightarrow\infty,
\end{equation}
\begin{equation}\label{c19}
H(x_0)\xi=E\xi,
\end{equation}
then
\begin{equation}\label{c20}
|\xi_n|<e^{-c|n|},\quad |n|\rightarrow\infty.
\end{equation}

We will first prove (\ref{c20}) for $|E|\leq C_0$. By Proposition \ref{p3.1},
there is $\Omega=\Omega_{N}(E)\subset\mathbb{T},\mathrm{mes}\Omega<e^{-\tilde{c}N^{\sigma}}$, such that if $x\notin \Omega$,
then for some $|m|<\sqrt{N}$,
\begin{equation}\label{c21}
|G_{[-N,N]}(x+m\omega,E)(n_1,n_2)|<e^{-c_{0}(|n_1-n_2|-\epsilon_{0}^{\frac{1}{40}}N)}, \quad |n_1|,|n_2|\leq N.
\end{equation}

Let
\begin{equation}\label{c22}
B(x)(n_1,n_2)=[\cos\pi(x+n_1\omega)][H_{[-N,N]}(x)-E](n_1,n_2),\quad n_1,n_2\in [-N,N]
\end{equation}
and $B_{n_1,n_2}(x)$ be the $(n_1,n_2)$-minor of $B(x)$. Then
\begin{equation}\label{c23}
|G_{[-N,N]}(x+m\omega,E)(n_1,n_2)|=|\cos\pi(x+m\omega+ n_1\omega)|\frac{|\det B_{n_1,n_2}(x+m\omega)|}{|\det B(x+m\omega)|}.
\end{equation}

Truncate power series for $\cos, \sin$ in (\ref{c23}), we may replace (\ref{c21}) by a polynomial of degree at most $N^4$.
Hence $\Omega$ may be assumed semi-algebraic of degree at most $N^5$.
Let $N_1=N^{C_{1}}$, $C_{1}$ is a sufficiently large constant. Then by Proposition \ref{p4.2},
\begin{equation}\label{c24}
\#\{|j|\leq N_1|x_{0}+j\omega\in\Omega\}<N_1^{1-\delta}, \quad\delta>0.
\end{equation}
Using (\ref{c24}), we may find an interval $I\subset[0,N_1]$ of size $N$ such that
\begin{equation}\label{c25}
x_{0}+j\omega\notin\Omega, \quad\forall j\in I \cup(-I).
\end{equation}

If $x_{0}+j\omega\notin\Omega$, then for some $|m_{j}|<\sqrt{N}$,
\begin{equation}\label{c26}
|G_{[a,b]}(x_{0},E)(n_1,n_2)|<e^{-c_{0}(|n_1-n_2|-\epsilon_{0}^{\frac{1}{40}}N)}, \quad n_1,n_2\in[a,b]
\end{equation}
where $[a,b]=[j+m_j-N,j+m_j+N]$.
By (\ref{c18}), (\ref{c19}), (\ref{c26}),
\begin{equation}\label{c27}
|\xi_j|\leq C \sum_{n\in[a,b],n'\notin[a,b]}e^{-c_{0}(|j-n|-\epsilon_{0}^{\frac{1}{40}}N)}e^{-\rho|n-n'|}|n'|
\leq C N_1e^{c_{0}\epsilon_{0}^{\frac{1}{40}}N}e^{-\frac{c_0}{2}N}<e^{-\frac{c_0}{3}N}.
\end{equation}
Denoting $j_0$ the center of $I$, we have
\begin{equation}\label{c28}
1=\xi_{0}\leq\|G_{[-j_0,j_0]}(x_0,E)\|\|R_{[-j_0,j_0]}H(x_0)R_{\mathbb{Z}\setminus[-j_0,j_0]}\xi\|.
\end{equation}

For $|n|\leq j_0$, by (\ref{c27}),
\begin{equation}\label{c29}
|(R_{[-j_0,j_0]}H(x_0)R_{\mathbb{Z}\setminus[-j_0,j_0]}\xi)_{n}|\leq\sum_{|n'|>j_0}e^{-\rho|n-n'|}|\xi_{n'}|
\end{equation}
\begin{equation*}
  \leq\sum_{j_0<|n'|\leq j_0+\frac{N}{2}}e^{-\rho|n-n'|}e^{-\frac{c_0}{3}N}+C\sum_{|n'|>j_0+\frac{N}{2}}e^{-\rho|n-n'|}|n'|
<Ce^{-\frac{c_0}{3}N}+CN_1e^{-\rho\frac{N}{2}}<e^{-\frac{c_0}{4}N}.
\end{equation*}

By (\ref{c28}), (\ref{c29}),
\begin{equation}\label{c30}
\|G_{[-j_0,j_0]}(x_0,E)\|>e^{\frac{c_0}{5}N},
\end{equation}
hence
\begin{equation}\label{c31}
{\rm dist}(E, {\rm spec} H_{[-j_0,j_0]}(x_0))<e^{-\frac{c_0}{5}N}.
\end{equation}

Denote
\begin{equation}\label{c32}
\mathcal{E}_{\omega}=\bigcup_{|j|\leq N_1}\left({\rm spec} H_{[-j_0,j_0]}(x_0)\cap[-2C_0,2C_0]\right).
\end{equation}
It follows from (\ref{c31}) that if $x\notin\bigcup\limits_{E'\in\mathcal{E}_{\omega}}\Omega(E')$,
then for some $|m|<\sqrt{N}$,
\begin{equation}\label{c33}
|G_{[-N,N]+m}(x,E)(n_1,n_2)|<e^{-c_{0}(|n_1-n_2|-\epsilon_{0}^{\frac{1}{40}}N)}, \quad n_1,n_2\in[-N,N]+m .
\end{equation}

Let $N_2=N^{C_{2}}$, $C_{2}$ is a sufficiently large constant.
Suppose
\begin{equation}\label{c34}
x_{0}+n\omega\notin\bigcup_{E'\in\mathcal{E}_{\omega}}\Omega(E'), \quad \forall \sqrt{N_2}<|n|<2N_2,
\end{equation}
then by (\ref{c33}), there are $|m_{n}|<\sqrt{N}$ such that
\begin{equation}\label{c35}
|G_{[-N,N]+n+m_{n}}(x_{0},E)(n_1,n_2)|<e^{-c_{0}(|n_1-n_2|-\epsilon_{0}^{\frac{1}{40}}N)}, \quad n_1,n_2\in[-N,N]+n+m_{n} .
\end{equation}

Let $\Lambda=\bigcup\limits_{\sqrt{N_2}<n<2N_2}([-N,N]+n+m_{n})\supset[\sqrt{N_2},2N_2]$. By Lemma \ref{l5.1},
\begin{equation}\label{c36}
|G_{\Lambda}(x_{0},E)(n_1,n_2)|<2e^{c_0\epsilon_{0}^{\frac{1}{40}}N},\quad n_1,n_2\in\Lambda,
\end{equation}
\begin{equation}\label{c37}
|G_{\Lambda}(x_{0},E)(n_1,n_2)|<e^{-\frac{c_{0}}{2}|n_1-n_2|}, \quad n_1,n_2\in\Lambda ,|n_1-n_2|>\frac{N_2}{10}  .
\end{equation}

For $\frac{1}{2} N_2 \leq j \leq N_2$, by (\ref{c36}), (\ref{c37}),
\begin{equation}\label{c38}
|\xi_{j}|\leq C\sum_{n\in\Lambda,n'\notin\Lambda}|G_{\Lambda}(x_{0},E)(j,n)|e^{-\rho|n-n'|}|n'|
\end{equation}
\begin{equation*}
\leq CN_2\sum_{|n-j|>\frac{N_2}{10}}e^{-\frac{c_{0}}{2}|n-j|}+CN_2\sum_{|n-j|\leq\frac{N_2}{10}}e^{c_0\epsilon_{0}^{\frac{1}{40}}N}e^{-\rho\frac{N_2}{4}}
\leq e^{-\frac{c_0}{40}N_2} \leq e^{-\frac{c_0}{40}j}.
\end{equation*}

Now we need to prove (\ref{c34}). Consider for $|j|\leq N_1 $, the set $S_j\subset\mathbb{T}^{2}\times\mathbb{R}$ of $(\omega,x,E')$ where
\begin{equation}\label{c39}
 \| k\omega\|>c|k|^{-A},\quad \forall 0<|k|\leq N,
\end{equation}
\begin{equation}\label{c40}
x\in\Omega(E'),
\end{equation}
\begin{equation}\label{c41}
E'\in{\rm spec} H_{[-j,j]}(x_0)\cap[-2C_0,2C_0].
\end{equation}

Let
\begin{equation}\label{c42}
S={\rm Proj}_{\mathbb{T}^{2}}S_{j}.
\end{equation}
Since ${\rm mes}\Omega(E')<e^{-\tilde{c}N^{\sigma}}$,
\begin{equation}\label{c43}
{\rm mes}S<N_1e^{-\tilde{c}N^{\sigma}}<e^{-\frac{1}{2}\tilde{c}N^{\sigma}}.
\end{equation}
Since $S_{j}$ is a semi-algebraic set of degree at most $N_1^5$, by Proposition \ref{p4.1},
$S$ is a semi-algebraic set of degree at most $N_1^{5C}$.

Take $n=1, B=N_1^{5C}, \eta=e^{-\frac{1}{2}\tilde{c}N^{\sigma}}, \epsilon =N_2^{-\frac{1}{10}}$
in Lemma \ref{l4.3}, we have $S=S_1\cup S_2$,
\begin{equation}\label{c44}
{\rm mes}{\rm Proj}_{\omega}S_1<B^{C}\epsilon<N_1^{C}N_2^{-\frac{1}{10}}<N_2^{-\frac{1}{11}}.
\end{equation}

We study the intersection of $S_{2}$ and sets
\begin{equation}\label{c45}
\{(\omega,x_0+n\omega)|\omega\in[0,1]\},\quad \sqrt{N_2}<|n|<2N_2,
\end{equation}
where $x_0+n\omega$ are considered mod 1. (\ref{c45}) lies in the parallel lines
\begin{equation}\label{c46}
L=L_{m}^{(n)}=\left[\omega=\frac{x}{n}\right]-\frac{m+x_{0}}{n}e_{\omega}, \quad |m|<N_2.
\end{equation}
Since $|{\rm Proj}_{L}e_{\omega}|<\frac{\epsilon}{100}$, by Lemma \ref{l4.3},
\begin{equation}\label{c47}
{\rm mes}(S_{2}\cap L)<B^{C}\epsilon^{-1}\eta^{\frac{1}{2}}<N_1^{C}N_2^{\frac{1}{10}}e^{-\frac{1}{4}\tilde{c}N^{\sigma}}.
\end{equation}
Summing over $n,m$,
\begin{equation}\label{c48}
{\rm mes}\{\omega\in[0,1]|(\omega,x_0+n\omega)\in S_{2},\exists\sqrt{N_2}<|n|<2N_2\}
<N_2^{2}N_1^{C}N_2^{\frac{1}{10}}e^{-\frac{1}{4}\tilde{c}N^{\sigma}}<e^{-\frac{1}{5}\tilde{c}N^{\sigma}}.
\end{equation}
From (\ref{c44}), (\ref{c48}), we exclude an $\omega$-set of measure $N_2^{-\frac{1}{11}}+e^{-\frac{1}{5}\tilde{c}N^{\sigma}}<N_2^{-\frac{1}{12}}$.
Summing over $|j|\leq N_1$, we get an $\omega$-set $\mathcal{R}_{N}, {\rm mes}\mathcal{R}_{N}<N_2^{-\frac{1}{13}}<N^{-10}$,
such that for $\omega\notin\mathcal{R}_{N}$,
\begin{equation}\label{c49}
|\xi_{j}|< e^{-\frac{c_0}{40}|j|},\quad |j|\in\left[\frac{1}{2}N^{C_2},N^{C_2}\right].
\end{equation}

Let
\begin{equation}\label{c50}
\mathcal{R}=\bigcap_{N_0\geq 1}\bigcup_{N\geq N_0}\mathcal{R}_{N},
\end{equation}
then ${\rm mes}\mathcal{R}=0$. If $\omega\notin\mathcal{R}$, then by (\ref{c50}),
there is $N_0\geq 1$ such that $\omega\notin\mathcal{R}_{N}, \forall N\geq N_0$.
By (\ref{c49}),
\begin{equation}\label{c51}
|\xi_{j}|< e^{-\frac{c_0}{40}|j|},\quad |j|\in\bigcup_{N\geq N_0}\left[\frac{1}{2}N^{C_2},N^{C_2}\right]=\left[\frac{1}{2}N_{0}^{C_2},\infty\right).
\end{equation}
This proves (\ref{c20}) for $|E|\leq C_0$. Note that in (\ref{c50}),
$\mathcal{R}=\mathcal{R}(C_0)$. Let
\begin{equation}\label{c52}
\tilde{\mathcal{R}}=\bigcup_{C_0\geq 1}\mathcal{R}(C_0),
\end{equation}
then ${\rm mes}\tilde{\mathcal{R}}=0$. We restrict $\omega\notin\tilde{\mathcal{R}}$.
This proves (\ref{c20}) for all $E\in\mathbb{R}$ and Theorem \ref{t5.2}.
\end{proof}

\subsection*{Acknowledgment}
This paper was supported by  National Natural
Science Foundation of China (No. 11790272 and No. 11771093).


\begin{thebibliography}{10}

\bibitem{B05}
J.~Bourgain.
\newblock {\em Green's function estimates for lattice {S}chr\"odinger operators
  and applications}, volume 158 of {\em Annals of Mathematics Studies}.
\newblock Princeton University Press, Princeton, NJ, 2005.

\bibitem{B07}
J.~Bourgain.
\newblock Anderson localization for quasi-periodic lattice {S}chr\"odinger operators on {$\Bbb Z^d$}, {$d$} arbitrary.
\newblock {\em Geom. Funct. Anal.}, 17(3):682--706, 2007.

\bibitem{BG}
J.~Bourgain and M.~Goldstein.
\newblock On nonperturbative localization with quasi-periodic potential.
\newblock {\em Ann. of Math. (2)}, 152(3):835--879, 2000.

\bibitem{BJ}
J.~Bourgain and S.~Jitomirskaya.
\newblock Anderson localization for the band model.
\newblock In {\em Geometric aspects of functional analysis}, volume 1745 of {\em Lecture Notes in Math.}, pages 67--79. Springer, Berlin, 2000.

\bibitem{BK}
J.~Bourgain and I.~Kachkovskiy.
\newblock Anderson localization for two interacting quasiperiodic particles.
\newblock {\em Geom. Funct. Anal.}, 29(1):3--43, 2019.

\bibitem{BGS}
J.~Bourgain, M.~Goldstein, and W.~Schlag.
\newblock Anderson localization for {S}chr\"odinger operators on {$\bold Z^2$} with quasi-periodic potential.
\newblock {\em Acta Math.}, 188(1):41--86, 2002.

\bibitem{BLS}
J.~B\'{e}llissard, R.~Lima  and E.~Scoppola.
\newblock Localization in {$v$}-dimensional incommensurate structures.
\newblock {\em Comm. Math. Phys.}, 88(4):465--477, 1983.

\bibitem{BPR}
S.~Basu, R.~Pollack and M.-F~Roy.
\newblock On the combinatorial and algebraic complexity of quantifier elimination.
\newblock {\em J.ACM}, 43(6):1002--1045, 1996.

\bibitem{GFP}
D.~Grempel, S.~Fishman and R.~Prange.
\newblock Localization in an incommensurate potential: an exactly solvable model.
\newblock {\em Phys. Rev. Lett.}, 49(11):833--836, 1982.

\bibitem{H}
R.~Han.
\newblock Shnol's theorem and the spectrum of long range operators.
\newblock {\em Proc. Amer. Math. Soc.}, 147(7): 2887--2897, 2019.

\bibitem{J}
S.~Jitomirskaya.
\newblock Metal-insulator transition for the almost {M}athieu operator.
\newblock {\em Ann. of Math. (2)}, 150(3):1159--1175, 1999.

\bibitem{JY}
S.~Jitomirskaya and F.~Yang.
\newblock Pure point spectrum for the Maryland model: a constructive proof.
\newblock {\em Ergodic Theory Dynam. Systems}, 2019.

\bibitem{JSY}
W.~Jian, Y.~Shi and X.~Yuan.
\newblock Anderson localization for one-frequency quasi-periodic block operators with long-range interactions.
\newblock {\em J. Math. Phys.}, 60(6):063504, 15, 2019.

\bibitem{MJ}
C. A.~Marx and S.~Jitomirskaya.
\newblock Dynamics and spectral theory of quasi-periodic {S}chr\"{o}dinger-type operators.
\newblock {\em Ergodic Theory Dynam. Systems}, 37(8):2353--2393, 2017.

\end{thebibliography}
\end{document}